\documentclass{amsart}

 \newtheorem{thm}{Theorem}[section]
 \newtheorem{cor}[thm]{Corollary}
 \newtheorem{lem}[thm]{Lemma}
 
 \theoremstyle{definition}
 
 \theoremstyle{remark}

 \numberwithin{equation}{section}

\begin{document}

\title[]
 {On the number of primitive and  quasi-primitive irreducible characters in finite groups}

%----------Author 1
\author{RIAD YOUMBAI}
\address{School of Mathematics and Statistics,\\
Central China Normal University,\\
Wuhan 430079, China}
\email{youmbair@mails.ccnu.edu.cn}

%----------Author 2
\author{GANG CHEN}
\address{School of Mathematics and Statistics,\\
Central China Normal University,\\
Wuhan 430079, China}
\email{chengangmath@mail.ccnu.edu.cn}


%\date{June 2, 2022}
\dedicatory{}
%%% ----------------------------------------------------------------------

\begin{abstract}
In this short note, it is proved that both the number of primitive characters and the number of quasi-primitive characters in a finite group $G$ is divisible by $|G:G'|$, where $G'$ is the derived subgroup of $G$. 
\end{abstract}

%%% ----------------------------------------------------------------------
\maketitle
%%% ----------------------------------------------------------------------
%\tableofcontents
\section{Introduction}
Throughout this short note, $G$ will be a finite group, $G'$ the derived subgroup of $G$, and ${\rm Irr}(G)$ the set of complex irreducible characters of $G$. All other notations are taken from \cite{Is1}. 

\medskip

Recall that an irreducible character $\chi\in {\rm Irr}(G)$ is {\it primitive} if $\chi$ can not be of the form $\theta^G$ for any character $\theta$ of a proper subgroup of $G$. And $\chi$  is called {\it quasi-primitive} if $\chi_N$ is a multiple of an irreducible character for every normal subgroup $N$ of $G$. The set of primitive irreducible characters and quasi-primitive irreducible characters of $G$ are respectively denoted by ${\rm Irr}_{\rm pri}(G)$ and ${\rm Irr}_{\rm qua}(G)$. 

\medskip

Let $H$ be a subgroup of $G$.  If $G'\subseteq H$, then ${\rm Irr}(G/H)$ acts on the set ${\rm Irr}(G)$; see \cite{chen} for more details about this action and its applications.  Furthermore, one can easily see that this action induces actions of ${\rm Irr}(G/H)$ on ${\rm Irr}_{\rm pri}(G)$ and ${\rm Irr}_{\rm qua}(G)$.    We denote the orbit and the stabilizer of an irreducible character $\chi$ of $G$ by ${\rm Orb}_H(\chi)$ and ${\rm Stab}_H(\chi)$,  respectively. For an irreducible character $\chi\in {\rm Irr}(G)$,  any element $x\in G$ satisfying $\chi(x)\ne 0$ is called a  {\it vanishing-off element} of $\chi$ and the subgroup ${\bf V}(\chi)$ generated by all vanishing-off elements  is called the {\it vanishing-off subgroup} of $\chi$.

\medskip

In this short note, we prove that the restriction of a quasi-primitive irreducible character to the derived subgroup is irreducible. Then we prove that both the number of primitive characters and the number of quasi-primitive characters in a finite group $G$ is divisible by $|G:G'|$, where $G'$ is the derived subgroup of $G$. 

\section{Main results and Proofs}
We start by giving the following lemma, which is \cite[Problem 6.2]{Is1}.

\begin{lem}
Let $G$ be a finite group and $H$ a subgroup containing $G'$. Let $\chi$ be an irreducible character of $G$ and write $Orb_H(\chi)=\{\lambda_1\chi,\lambda_2\chi,...,\lambda_k\chi\}$ for some $\lambda_i\in {\rm Irr}(G/H)$. Let $\xi$ be an irreducible constituent of $\chi_H$, then
$$
e(\lambda_1\chi+\lambda_2\chi+...+\lambda_k\chi)=\xi^G
$$ for some positive integer $e$. In particular, $\lambda_1\chi+\lambda_2\chi+...+\lambda_k\chi$ vanishes on $G-H$.

\end{lem}

\begin{cor}

Let $G'\subseteq H\subseteq G$ and $\xi\in {\rm Irr}(H)$, then the set of irreducible constituents of $\xi^G$ forms an orbit under the action of ${\rm Irr}(G/H)$ on ${\rm Irr}(G)$.
\end{cor}

The following is \cite[Lemma 7.1 (c)]{Is2}. We give a proof for the reader's convenience. 

\begin{thm}\label{1150c}
Let $\chi\in {\rm Irr}(G)$ and $G'\subseteq H$ be a subgroup of $G$. Then 
$$
H{\bf V}(\chi)=G \text{ if and only if }\chi_{H} \text{ is irreducible.}
$$
Furthermore, if one of the equivalent condition happens, then $	{\rm Stab}_H(\chi)=\{1_G\}$. 

\end{thm}

\begin{proof} Assume that  $H{\bf V}(\chi)=G$ and consider the action of ${\rm Irr}(G/H)$ on ${\rm Irr}(G)$. 	We claim that $$
{\rm Stab}_H(\chi)=\{1_G\}. 
$$ 
	To see this, suppose $\lambda\chi=\chi$ for some $\lambda\in {\rm Irr}(G/H)$. Let $g\in G=H{\bf V}(\chi)$, one can write $g=\prod_{i=1}^n x_i$ for some $x_1,x_2,...,x_n$, where $x_i$ or $x_i^{-1}$ is a non-vanshing element of $\chi$ or belongs to $H$. It follows that $\lambda(x_i)=1$ for all $i$ and so $\lambda(g)=1$. Thus we have $\lambda=1_G$ as claimed.
	
	\medskip
	
Let $\xi$ be an irreducible constituent of $\chi_H$ and $e=[\chi_H,\xi]$. Since the orbit of $\chi$ contains $|{\rm Irr}(G/H):\{1_G\}|=|{\rm Irr}(G/H)|=|G:H|$ elements, we have
$$|G:H|\xi(1)=\xi^{G}(1)=e|G:H|\chi(1)$$ where the second equality follows by Lemma 2.1. Since $\chi(1)\geq \xi(1)$, we have  $e=1$ and $\chi(1)=\xi(1)$. Thus,  $\chi_H$ is irreducible.

\medskip

Conversely, assume that  $\chi_H$ is irreducible. Towards a contradiction, suppose 
that $H{\bf V}(\chi)\ne  G$.  Denote $H{\bf V}(\chi)$ by $K$. Then $\chi$ is zero in $G-K$ and one can see that $\lambda\chi=\chi$ for any $\lambda\in {\rm Irr}(G/K)$. 

\medskip

By Lemma 2.1 we have $a\chi=\xi^G$ for some positive integer $a$ and some irreducible character $\xi$ of $K$. Since $\chi_{H}$ is irreducible, so is $\chi_K$.  Write $\chi_K=\xi$. It follows that $a=1$ and 
$$
\chi(1)=\xi^G(1)=|G:K|\xi(1)=|G:K|\chi(1),
$$
which is a contradiction as we are assuming that $K$ is proper in $G$. The last statement follows by the first paragraph of the proof. 
\end{proof}

\medskip

The next Lemma is Corollary 11.22 of \cite{Is1}. 
\begin{lem} Let $H$ be a normal subgroup of $G$ with $G/H$ cyclic and let $\varphi\in {\rm Irr}(H)$ be invariant in $G$. Then $\varphi$ is extendible to $G$.
\end{lem}

\medskip

The following is the first main result of this short note. 

\begin{thm}\label{1149b}
Let $\chi$ be a quasi-primitive irreducible character of $G$. Then $\chi_{G^{'}}$ is irreductible.
\end{thm}

\begin{proof}
Let $\chi$ be as in the assumption. We claim that ${\bf V}(\chi)G'=G$.  

\medskip

If this is false, one can choose ${\bf V}(\chi)G'\subseteq K\subseteq G$ with $|G:K|$ a prime.
Since $\chi(g)=0$ for all $g\in G-K$, $Orb_{K}(\chi)=\{\chi\}$ and so by Lemma 2.1, we have $a\chi=\xi^G$ for some positive integer $a$ and an irreducible character $\xi$ of $K$.

\medskip

 Since $\chi$ is quasi-primitive, $\xi$ is invariant in $G$ and thus is extandible by Lemma 2.6. It follows that $a=1$ and 
$$
\chi(1)=\xi^G(1)=|G:K|\xi(1)=|G:K|\chi(1).
$$
which is a contradiction as we are assuming that $K$ is proper in $G$. Thus,  ${\bf V}(\chi)G'=G$ and the result follows by Theorem 2.3.
\end{proof}

Since every primitive irreducible character is quasi-primitive by \cite[Corollary 6.12]{Is1}, the following corollary follows by Theorem \ref{1149b}.

\begin{cor}\label{1148a}
Let $\chi$ be a primitive irreducible character of a finite group $G$. Then $\chi_{G^{'}}$ is irreducible.
\end{cor}

\begin{cor} Metabelian groups are M-groups.
\end{cor}

\begin{proof} Let $\chi$ be a nonlinear irreducible character of the metabelian group $G$. Since $G'$ is abelian, the restriction of $\chi$ to $G'$ is reducible, and so by Corollary \ref{1148a}, $\chi$ is not primitive. Let $H<G$ be minimal such that there exists $\eta\in {\rm Irr}(H)$ with $\eta^G=\chi$. By transitivity of the induction, $\eta$ is primitive. Since subgroups of metabelian groups are metabelian, we conclude that $\eta$ is linear.

\end{proof}

\medskip

The following is the second main result of this short note.

\begin{thm}Let $G$ be a finite group.  Then ${\rm gcd}(|{\rm Irr}_{\rm pri}(G)|, |{\rm Irr}_{\rm qua}(G)|)$ is divisible by $|G:G'|$.
\end{thm}

\begin{proof} By Theorems \ref{1150c} and \ref{1149b} and Corollary \ref{1148a}, one can see that the action of ${\rm Irr}(G/G')$ on ${\rm Irr}_{\rm pri}(G)$ is semiregular and so is on ${\rm Irr}_{\rm qua}(G)$. Now the result follows by fundamental counting principle. 
	
\end{proof}

 \section{Final Remarks}
 
 It is known in non solvable groups that quasi-primitive characters are not necessarily primitive. Similarly, the converse of \ref{1149b} is not true. For instance an irreducible character of $S_4$ of degree 3 is not quasi-primitive and yet its restriction to $S_4^{'}=A_4$ is irreducible. Let $G$ be a finite group and denote the set $\{\chi\in {\rm Irr}(G),\text{ }{\bf V}(\chi)G'=G\}$ by ${\rm Irr}_{\rm G^{'}}(G)$. By Theorems \ref{1150c} and \ref{1149b}, we have
 $$
 {\rm Irr}_{\rm pri}(G)\subseteq {\rm Irr}_{\rm qua}(G)\subseteq {\rm Irr}_{\rm G^{'}}(G). 
 $$

\medskip

\section*{Acknowledgment}
The first author is supported by the Algerian Ministry of Higher Education and Scientific Research Scholarship and China Scholarship Council (CSC). The second author is supported by Natural Science Foundation of China (No. 11971189, No. 12161035)

% ------------------------------------------------------------------------
\end{document}